\documentclass[12pt]{article}
\usepackage{amsthm,amsmath,amsfonts,amssymb}
\usepackage{verbatim}
\usepackage[alphabetic]{amsrefs}
\usepackage{graphicx}

\newtheorem{thm}{Theorem}
\newtheorem{lemma}[thm]{Lemma}
\newtheorem{cor}[thm]{Corollary}
\newtheorem{prop}[thm]{Proposition}
\newtheorem{definition}[thm]{Definition}

\newcommand{\R}{\mathbb{R}}

\newcommand{\dd}{\partial}
\newcommand{\e}{\varepsilon}
\newcommand{\lap}{\Delta}
\renewcommand{\div}{\operatorname{div}}

\DeclareMathOperator*{\osc}{osc}

\begin{document}
\title{Regularity of the $p$-Poisson equation in the plane}
\author{Erik Lindgren\footnote{erik.lindgren@math.ntnu.no} \qquad Peter Lindqvist\footnote{lqvist@math.ntnu.no}}
\date{Department of Mathematical Sciences\\Norwegian University of Science and Technology\\NO-7491
  Trondheim, Norway}
\maketitle
\begin{abstract}\noindent 
 \textsf{We study the regularity of the $p$-Poisson equation
 $$ 
 \lap_p u = h, \quad h\in L^q
 $$ in the plane. In the case $p>2$ and $2<q<\infty$ we obtain the sharp H\"older exponent for the \emph{gradient}. In the other cases we come arbitrarily close to the sharp exponent.} 
      \end{abstract}
\section{Introduction} In the plane, the theory of many partial differential equations is more explicit than in higher dimensions. Sometimes the theory of quasiregular mappings and other devices are available, see \cite{Ber58}. Our object is to study the so-called $p$-Poisson equation
\begin{equation}\label{eq:ppoission}
\lap_p v\,(x,y)\equiv \div(|\nabla v(x,y)|^{p-2}\nabla v(x,y)) = h(x,y), 
\end{equation}
in a bounded domain $\Omega\subset \R^2$, when $1<p<\infty$. This equation arises as the Euler-Lagrange equation of the variational integral 
$$
\iint_\Omega \left(\frac1p |\nabla v|^p +hv\right) dx dy.
$$
The weak solutions are known to be of class $C^{1,\kappa}_\text{loc}$ for some $\kappa\in (0,1)$. We are interested in \emph{the sharp H\"older exponent for the gradient} of the solution. We record a well known result:
\begin{prop} \label{prop:c1kappa} Suppose that $v$ is a solution of \eqref{eq:ppoission} in the disc $B_{2R}$ and that $h\in L^q(B_{2R})$ for some fixed $2<q\leq\infty$. Then  $v\in C^{1,\kappa}_\text{loc}(B_{2R})$, for some $\kappa=\kappa(p,q)$. We have the estimate
$$
\| v\|_{C^{1,\kappa}(B_R)}\leq A, 
$$
where $A=A(p,q,R,\|h\|_{L^q(B_{2R})}, \|v\|_{L^\infty(B_{2R})})$.
\end{prop}
Here and in the rest of the paper, we use the notation
$$
[u]_{C^s(D)}=\left\|\frac{u(x)-u(y)}{|x-y|^s}\right\|_{L^\infty(D\times D)}, \quad \|u\|_{C^s(D)}=[u]_{C^s(D)}+\|u\|_{L^\infty(D)}
$$
and
$$
\|u\|_{C^{1,s}(D)}=\|\nabla u\|_{C^s(D)}+\|u\|_{L^\infty(D)}
$$
when $s\in (0,1)$ and $D$ is a bounded domain. The proof of the above theorem can for $q=\infty$ be found in \cite{Tol84} and for $2<q<\infty$ in \cite{Lie93}. See also the corollary on page 830 in \cite{DiB83}.

In the homogeneous case, $\lap_p v = 0$, the optimal H\"older exponent 
$$
\kappa = \frac16\left(\frac{p}{p-1}+\sqrt{1+\frac{14}{p-1}+\frac{1}{(p-1)^2}}\right)>\min\left(p-1,\frac{1}{p-1}\right),\quad (p\neq 2)
$$
has been determined by Iwaniec and Manfredi in \cite{IM89}. They used the hodograph transform. However, the ``torsional creep equation''
$$
\lap_p v = 2,
$$
studied for example in \cite{Kaw90}, has a weak solution given by 
$$
v(x)=\frac{p-1}{p}|x|^\frac{p}{p-1}, 
$$
so that $|\nabla v(x)|=|x|^\frac{1}{p-1}$, exhibiting the fact that, in general one must have $\kappa\leq \frac{1}{p-1}$ if $q=\infty$. The example
$$
v(x)=\int_0^{|x|}\left(\frac{\rho^{1-\frac2q}}{(\ln\rho)^\frac{2}{q}}\right)^\frac{1}{p-1} d\rho
$$
solves the $p$-Poisson equation with the right-hand side in $L^q$, showing that $\kappa\leq \frac{1-\frac{2}{q}}{p-1}$ when $2<q<\infty$. 

The exponent $\beta$ defined as follows plays a crucial role:
\begin{definition}\label{det:beta}
\noindent {\bf Case $ 1<p<2$:} If $q=\infty$ let $\beta$ be any number less than 2 and if $q<\infty$ let 
$$\beta= 2-\frac{2}{q}.
$$
\noindent {\bf Case $ p>2$:} If $q=\infty$ let $\beta$ be any number less than $p/(p-1)$ and if $q<\infty$ let 
$$\beta=\frac{p-\frac{2}{q}}{p-1}.$$
\end{definition}
In the theorem below we determine the \emph{optimal} H\"older exponent at least in the case $p>2$ and $2<q<\infty$. It is our main result.

\begin{thm}\label{thm:main} Suppose $\lap_p v = h$ in $\Omega$ and that $h\in L^q(\Omega)$, where $2< q\leq \infty$. Then $\nabla v \in C^{\beta-1}_\text{loc}(\Omega)$.
In particular, for any compact $K\subset\Omega$, we have the estimate
$$
[\nabla v]_{C^{\beta-1}(K)}\leq C(q,p,\beta,K)\max \left(\|h\|_{L^q(\Omega)}^\frac{1}{p-1},\|v\|_{L^\infty(\Omega)}\right).
$$ 
\end{thm}
It is worth our while to mention that for a bounded right-hand $h$ side we obtain for the estimate
$$
|\nabla v(x)-\nabla v(y)|\leq L_\e|x-y|^{\frac{1}{p-1}-\e},\quad p>2,
$$
for every $\e>0$. It is likely that it also holds when $\e=0$.
Our method of proof is based on universal estimates for the $p$-Laplace equation $\lap_p u =0$, which come from the fact that the complex gradient, $f=u_x-iu_y$ is a quasiregular mapping. For this method it is essential that the right-hand side is zero. A balanced perturbation of the  $p$-Poisson equation leads to the $p$-Laplace equation at the limit so that the universal estimates can be employed.

\emph{Acknowledgments:} This work was written at the Mittag-Leffler Institute. The topic was inspired by a talk of J. M. Urbano concerning \cite{TU13}. G. Mingione has informed us that alternative proofs can be extracted from various estimates in \cite{KM12} and \cite{KM13}. The authors are also truly grateful to J. Lewis for reading the proof at an early stage.

\section{Auxiliary results for the homogeneous equation $\lap_p u=0$}
It was proved by Bojarski and Iwaniec that the complex gradient
$$
f=u_x-iu_y
$$
of a solution to the $p$-Laplace equation $\lap_p u =0$ is a quasiregular mapping; see \cite{BI87}. We need the following consequence of this fundamental result.

\begin{lemma}\label{lem:qreg} Let $p\geq 2$. Suppose $u\in C(B_{2R})\cap W^{1,p}(B_{2R})$ is a solution to $\lap_p u =0$ in the disc $B_{2R}$. Then there is a constant $\Lambda = \Lambda(p)$ such that 
$$  
[\nabla u]_{C^\alpha(B_R)}\leq \frac{\Lambda}{R^{1+\alpha}}\|u\|_{L^\infty(B_{2R})}, 
$$
where $\alpha = \frac{1}{p-1}$.
\end{lemma}

It is of utmost importance that the same $\Lambda$ will do for all solutions $u$. We sketch the proof of this known result.
\begin{proof}[~Sketch of the proof]
 First, by \cite{BI87} the complex gradient $f(z)$, which belongs to $W^{1,2}_\text{loc}(\Omega)$ and is continuous, satisfies the inequality
 $$
 \Big|\frac{\dd f}{\dd \bar z}\Big|\leq \frac{p-2}{p}\Big|\frac{\dd f}{\dd z}\Big|, 
 $$
 a.e. in the $\Omega$. Here it is decisive that $(p-2)/p<1$. As in the proof of Lemma 12.1 in \cite{GT01} it follows that the Dirichlet integral 
 $$
 I(r)=\iint_{B_r}|D f|^2dx dy
 $$
 satisfies the inequality
 $$
 I(r)\leq I(r_0)\left(\frac{r}{r_0}\right)^{2\alpha}, \quad \alpha = \frac{1}{p-1},
 $$
 when $r\leq r_0<2R$. Then Morrey's lemma implies
 \begin{equation}\label{eq:morest}
 |f(z_2)-f(z_1)|\leq 2\left(\frac{|z_2-z_1|}{r_0}\right)^ \alpha\sqrt{\frac{1}{\alpha} I(r_0)}, 
 \end{equation}
 when $|z_2-z_1|\leq r_0<2R$; see Lemma 12.2 in \cite{GT01}.
 
 We also have the standard estimate
 \begin{equation}\label{eq:Iest}
 I(r)=\iint_{B_r} |Df|^2 dx dy \leq \frac{C_1}{r^2}\iint_{B_{\frac{3r}{2}}} |f|^2 dx dy=\frac{C_1}{r^2}\iint_{B_{\frac{3r}{2}}} |\nabla u|^2 dx dy,
 \end{equation}
 for a quasiregular mapping, sometimes called Mikljukov's inequality. There $C_1$ depends on the dilatation $1/(p-1)$, hence only on $p$. Now 
 \begin{align}\label{eq:cacc}\nonumber 
\left( \frac{1}{r^2}\iint_{B_{\frac{3r}{2}}} |\nabla u|^2 dx dy\right)^\frac12&\leq \left( \frac{1}{r^2}\iint_{B_{\frac{3r}{2}}} |\nabla u|^p dx dy\right)^\frac1p\\
\leq \left( \frac{C_2}{r^{p+2}}\iint_{B_{2r}} |u|^p dx dy\right)^\frac1p& \leq \frac{C_2^\frac1p}{r}\|u\|_{L^\infty(B_{2r})},
 \end{align}
 by H\"older's inequality and a standard Caccioppoli estimate. The new constant $C_2$ depends only on $p$. Combining \eqref{eq:Iest} and \eqref{eq:cacc} we arrive at
 $$
 |f(z_2)-f(z_1)|\leq \frac{\Lambda}{r_0} \left(\frac{|z_2-z_1|}{r_0}\right)^\alpha\|u\|_{L^\infty(B_{2r_0})},
 $$
 whenever $|z_2-z_1|\leq r_0<R$. The various constants have been joined in $\Lambda$. This is the desired result.
\end{proof}

The above lemma has the following immediate consequence.
\begin{cor}[Liouville]{\label{cor:liouville}}  Let $p>2$. If $\lap_p u =0$ in $\R^2$ and if 
$$
\|u\|_{L^\infty(B_{R_j})}\leq CR_j^{1+\alpha-\e}, \quad \alpha = \frac{1}{p-1}
$$ 
for some fixed constant $C$, some subsequence $R_j\to\infty$ and $\e>0$, then $\nabla u$ must be constant.
\end{cor}
Through the conjugate function we can deduce the following result:
\begin{lemma}\label{lem:qreg2}
 Let $1<p<2$ and suppose that $\lap_p u =0$ in $B_{2R}$ and that $|\nabla u (0)|=0$. Then there is a constant $\Lambda=\Lambda(p)$ such that 
$$
\frac{ | \nabla u(x)|}{|x|}\leq \frac{\Lambda}{R^2}\|u\|_{L^\infty(B_{2R})}
 $$
for any $x\in B_R$.
\end{lemma}
\begin{proof} The conjugate function $v$, given by
$$
\left\{\begin{array}{lr}
v_x = -|\nabla u|^{p-2}u_y,\\
v_y = |\nabla u|^{p-2}u_x, 
\end{array}\right.
$$ 
satisfies
$$
\left\{\begin{array}{lr}
\lap_{p'} v  = 0,\\
|\nabla v|^{p'} =|\nabla u|^p,
\end{array}\right.
$$ 
where $1/{p'}+1/p=1$ so that $p'=p/(p-1)>2$, see \cite{AL88} . From \eqref{eq:morest} in the proof of Lemma \ref{lem:qreg} we can read off
$$
 |\nabla v(x)|\leq 2\left(\frac{|x|}{r}\right)^\frac{1}{p'-1}\sqrt{(p'-1) I(r)}, 
 $$
 when $|x|\leq r<2R$. Here
$$
I(r)=\iint_{B_r}|D^2 v|^2 dx dy \leq \frac{C_1}{r^2}\iint_{B_\frac{3r}{2}} |\nabla v|^2 dx dy,
$$
by Mikljukov's inequality. By H\"olders inequality and a standard Caccioppoli estimate it follows that
\begin{align*}
 \sqrt{I(r)}\leq \left(\frac{C_2}{r^2}\iint_{B_\frac{3r}{2}}|\nabla v|^{p'} dx dy\right)^\frac{1}{p'}
 &=\left(\frac{C_2}{r^2}\iint_{B_\frac{3r}{2}}|\nabla u|^p dx dy\right)^\frac{1}{p'}\\
 &\leq \left(\frac{C_3}{r^{2+p}}\iint_{B_{2r}}|u|^p dx dy\right)^\frac{1}{p'}\\
 &\leq C_4\left(\frac{\|u\|_{L^\infty(B_{2r})}}{r}\right)^\frac{p}{p'}.
\end{align*}
Recalling that $|\nabla u|^{p-1}=|\nabla v|$ we obtain
$$
\frac{|\nabla u(x)|}{|x|}\leq \Lambda\frac{\|u\|_{L^\infty(B_{2r})}}{r^2},
$$
for $|x|\leq r<2R$. This implies the desired result.
\end{proof}
As a consequence we also obtain a Liouville-type result when $1<p<2$.
\begin{cor}[Liouville]\label{cor:liouville2} Assume $1<p<2$. If $\lap_p u =0$ in $\R^2$, $|\nabla u(0)|=0$ and 
$$
\|u\|_{L^\infty(B_{R_j})}\leq CR_j^{2-\e},
$$ 
for some fixed constant $C$, some subsequence $R_j\to\infty$ and $\e>0$, then $ u$ must be constant.

\end{cor}
\section{The oscillation of the solution when the gradient is small}
In this chapter we consider the equation
$$
\lap_p v = h\quad \text{in $B_1$}
$$
under the assumptions
$$
\|h\|_{L^q(B_1)}\leq 1,\,q>2,\quad \|v\|_{L^\infty(B_1)}\leq 1.
$$
In this normalized situation, our aim is to prove the following estimate:
\begin{prop}\label{prop:supprop} Let $\beta$ be as in Definition \ref{det:beta}. If the inequality $|\nabla v(x)|\leq r^{\beta-1}$, where $r<1/4$, holds at some fixed point $x\in B_{1/2}$, then
$$
S_r=\sup_{y\in B_r(x)}|v(y)-v(x)|\leq Cr^\beta,
$$
where $C=C(p,q,\beta)$.
\end{prop}
The difficulty is that the gradient constraint is only assumed at the point $x$, otherwise the result would be trivial. The proof is based on rescaled functions and a blow-up argument. At the end, the limiting function turns out to be a solution of the $p$-Laplace equation in the whole plane, which satisfies the Liouville theorem. We begin with the key lemma.

\begin{lemma}\label{lem:keylem} Assume the hypotheses of Proposition \ref{prop:supprop}. Then there is a constant $C=C(p,q,\beta)$ such that for every fixed $r\in (0,1/4)$, at least one of the following alternatives hold:
\begin{enumerate}
\item[(i)]
$S_r=\displaystyle\sup_{y\in B_r(x)}|v(y)-v(x)|\leq Cr^\beta,$
\item[(ii)] There is an integer $k\geq 1$ such that $2^kr<1/4$ and $S_r\leq 2^{-k\beta}S_{2^kr}$.
\end{enumerate}
\end{lemma}
\begin{proof}
The proof is indirect, starting from the antithesis that no constant $C$ will ever do. Thus, giving $C$ the successive values $j=1,2,3,\ldots$, we can select solutions $v_j$, radii $r_j<1/4$ and points $x_j\in B_{1/4}$ so that the three conditions
\begin{enumerate}
\item[1)] $\displaystyle S_{r_j}=\sup_{y\in B_{r_j}(x_j)}|v_j(y)-v_j(x_j)|\geq jr_j^\beta$,
\item[2)] $S_{r_j}\geq 2^{-k\beta}S_{2^kr_j}$ for all integers $k$ such that $2^kr_j<1/4$, or $r_j\geq\frac18$,
\item[3)] $|\nabla v_j(x_j)|\leq r_j^{\beta-1}$,
\end{enumerate}
all hold. By 1) and the assumed bound on $v_j$, we have $jr_j^\beta\leq 2$, which forces $r_j\to 0$, as $j\to\infty$. This excludes the alternative $r_j\geq 1/8$ in 2). Notice that 2) is perfectly designed for iteration. We define the rescaled functions
$$
V_j(x)=\frac{v_j(x_j+r_jx)-v_j(x_j)}{S_{r_j}}, \quad j=1,2,3,\ldots
$$
which, as we will see, solve a $p$-Poisson equation. By the chain rule
$$
\nabla V_j(x)=\frac{r_j}{S_{r_j}}\nabla v_j(y)\Big|_{y=x_j+r_jx}.
$$
The following properties are now immediate:
$$
\left\{\begin{array}{lr}
 V_j(0)=0,\\
 |\nabla V_j(0)|=\frac{r_j}{S_{r_j}}|\nabla v_j(x_j)|\leq \frac{r_j^\beta}{S_{r_j}}\leq \frac{1}{j}\to 0,\\
 \sup_{B_{2^k}} |V_j|=\frac{S_{2^kr_j}}{S_{r_j}}\leq 2^{k\beta},\text{ for all integers $k$ such that $2^k<\frac{1}{4r_j}$}, \\
 \lap_p V_j(x)=\frac{r_j^p}{S_{r_j}^{p-1}}h(x_j+r_jx)\equiv H_j(x), \text{ when }|x|<\frac{1}{4r_j}.
\end{array}\right.
$$
In particular, the rescaled functions $V_j$ solve a $p$-Poisson equation in the disc $|x|<1/(4r_j)$, which is expanding to the whole plane as $j\to \infty$. Note that the use of second derivatives can be totally avoided if one just writes the equations in their weak form, using test functions under the integral sign.

Recall that $2<q\leq \infty$. We need to treat the case $q=\infty$ separately in the following formal computations.\\

\noindent {\bf Case $q=\infty$:} Now $p-\beta(p-1)>0$ and thus for any $R>0$ we have
\begin{align*}
\|\lap_p V_j\|_{L^\infty(B_R)}&=\frac{r_j^p}{S_{r_j}^{p-1}}\|h\|_{L^\infty(B_{Rr_j}(x_j))}\\
&\leq \frac{r_j^p}{(j r_j^\beta)^{p-1}}\to 0,
\end{align*}
as $j\to\infty$, since sooner or later $Rr_j<1/2$, as required.\\

\noindent {\bf Case $q<\infty$:} Now $q(p-\beta(p-1))-2\geq 0$ and 
\begin{align*}
\|\lap_p V_j\|_{L^q(B_R)}^q&=\frac{r_j^{pq}}{S_{r_j}^{(p-1)q}}\int_{B_R}|h(x_j+r_jx)|^q\,dx\\
&=\frac{r_j^{pq-2}}{S_{r_j}^{(p-1)q}}\int_{B_{Rr_j}(x_j)}|h(y)|^q\,dy\leq \frac{r_j^{pq-2}}{S_{r_j}^{(p-1)q}}\\
&\leq \frac{r_j^{q(p-\beta(p-1)-2}}{(j)^{(p-1)q}}\to 0, 
\end{align*}
as $j\to\infty$, since as above, $Rr_j<1/2$ sooner or later, as required.

Now we go back to the equation for the $V_j$s:
$$
\lap_p V_j = H_j.
$$
In order to be able to pass to the limit as $j\to\infty$, we need some compactness. We recall Proposition \ref{prop:c1kappa} in the introduction. It yields an estimate of the form
\begin{equation}
\label{eq:vjuni}
\| V_j\|_{C^{1,\kappa}(B_\frac{R}{2})}\leq A(p,q,R,\|H_j\|_{L^q(B_R)}, \|V_j\|_{L^\infty(B_R)}), 
\end{equation}
for some $\kappa=\kappa(p,q)$. Recall also that
$$
\|V_j\|_{L^\infty(B_{2^k})}\leq 2^{\beta k}
$$
and that
$$
\|H_j\|_{L^q(B_R)}<1, \text{ for $j$ large enough}.
$$
Thus, the bound in \eqref{eq:vjuni} is uniform in $j$. It follows that, up to extracting a subsequence, $V_j$ converges locally uniformly in $C^{1,\kappa/2}(\R^2)$ to some limit function $V$. The limit function inherits many properties. We obtain that
$$
\left\{\begin{array}{lr}
V(0)=0,\quad |\nabla V(0)|=0,\\
\sup_{B_{2^k}}|V|\leq 2^{k\beta}\text{ for all integers $k\geq 1$},\\
\sup_{B_1}|V|=1,\\
\lap_p V = 0\text{ in $\R^2$}.
\end{array}\right.
$$
Thus, $V$ is an entire solution of the $p$-Laplace equation and Liouville's theorem applies. Since in any case $\beta<2$ if $1<p<2$ and $\beta<p/(p-1)$ if $p>2$, it follows from Corollary \ref{cor:liouville} and Corollary \ref{cor:liouville2} with $R_j=2^j$, that $\nabla V$ reduces to a constant. Thus, $V$ is an affine function and since $V(0)=|\nabla V(0)|=0$, we must have $v\equiv 0$. This contradicts the fact that 
$$
\sup_{B_1} |V| =1.
$$
We conclude that the antithesis is false. The lemma follows.
\end{proof}
In order to prove Proposition \ref{prop:supprop} we have to show that the first alternative in Lemma \ref{lem:keylem} is always valid.

\begin{proof}[~Proof of Proposition \ref{prop:supprop}] If alternative (i) holds for all $r<1/4$ we are done. If not, we pick a radius $r$ for which, by alternative (ii), 
$$
S_r\leq 2^{-k_1\beta}S_{2^{k_1}r},
$$
for some integer $k_1$ with $2^{k_1}r<1/4$. If (i) holds for $S_{2^{k_1}r}$, then
$$
S_{r}\leq 2^{-k_1\beta}S_{2^{k_1}r}\leq 2^{-k_1\beta}C(2^{k_1}r)^\beta=Cr^\beta
$$
and again we are done. If not, we continue with 
$$
S_{2^{k_1}r}\leq 2^{-k_2\beta}S_{2^{k_2}2^{k_1}r}, 
$$
where $2^{k_1}2^{k_1}r<1/4$. Iterating this as long as alternative (i) fails, we obtain
$$
S_r\leq 2^{-k_n\beta}\cdots 2^{-k_1\beta}S_{2^{k_n}\cdots 2^{k_1}r}=2^{-(k_1+\cdots +k_n)\beta}S_{2^{k_1+\cdots +k_n} r}, 
$$
where $2^{k_1+\cdots +k_n} r<1/4$. Since every $k_j\geq 1$, the procedure must stop after a finite number of steps (depending on $r$), at its latest when 
$$2^{k_1+\cdots+ k_n}r\geq \frac18 .
$$
Then the alternative (i) holds for the radius $2^{k_1+\cdots +k_n} r$ and so, finally, 
$$
S_r\leq 2^{-(k_1+\cdots+ k_n)\beta}S_{2^{k_1+\cdots+ k_n} r}\leq 2^{-(k_1+\cdots+ k_n)\beta}C(2^{k_1+\cdots+ k_n} r)^\beta= Cr^\beta.
$$
This proves the claim.
\end{proof}
\section{Proof of the main theorem}
We are now ready to give the proof of our main result. The idea is that when the gradient is small, we can apply the result of the previous section to obtain the desired estimates. On the other hand, when the gradient is large, then the equation becomes non-degenerate so that classical results apply. We first prove the following intermediate result.
\begin{thm}\label{thm:mainprop} Let $\beta$ be as in Definition \ref{det:beta}. Assume that 
$$
\lap_p  v = h \text{ in $B_1$},\quad \|v\|_{L^\infty(B_1)}\leq 1,\quad \|h\|_{L^q(B_1)}\leq 1,
$$
for some $q>2$. Then for any $x\in B_{1/2}$, 
\begin{equation}\label{eq:betaest}
\sup_{B_r(x)}|v(y)-v(x)-(y-x)\cdot \nabla v(x)|\leq Lr^\beta, 
\end{equation}
when $0<r<1/4$ and where $L=L(p,q,\beta)$.
\end{thm}
\begin{proof} Fix $x\in B_{1/2}$. If $|\nabla v(x)|\leq r^{\beta-1}\leq\left(\frac14\right)^{\beta-1}$, then by Proposition \ref{prop:supprop},
\begin{equation}\label{eq:levelest}
\sup_{B_r(x)}|v(y)-v(x)-(y-x)\cdot \nabla v(x)|\leq Cr^\beta+r\cdot r^{\beta-1}=(C+1)r^\beta,
\end{equation}
where $C$ depends on $p$, $q$ and $\beta$. We need the estimate also for $r^{\beta-1}<|\nabla v(x)|$. To this end, let $\rho =|\nabla v(x)|^\frac{1}{\beta-1}>0$ and 
$$
 w(y)=\frac{v(x+\rho y)-v(x)}{\rho^\beta},
$$
so that  $\nabla w(y)=\rho^{1-\beta}\nabla v(x+\rho y)$ and
$$
\sup_{B_1}|w(y)|=\rho^{-\beta}\sup_{B_\rho(x)}|v(y)-v(x)|\leq C,
$$
since $\rho$ is the largest radius for which Proposition \ref{prop:supprop} is available. Moreover, by calculation
\begin{align*}
\|\lap_p w\|_{L^q(B_1)}&\leq \|h\|_{L^q(B_\rho(x))}\rho^{p-\beta(p-1)-\frac{2}{q}}\\
&\leq \|\nabla v \|_{L^\infty(B_\frac12)}^{p-\beta(p-1)-\frac{2}{q}}\leq D=D(p,q,\beta),
\end{align*}
by Proposition \ref{prop:c1kappa}. We once more apply Proposition \ref{prop:c1kappa} to obtain the estimate
$$
\|w\|_{C^{1,\kappa}(B_\frac12)}\leq A=A(p,q,\beta),\quad \kappa=\kappa(p,q).
$$
Therefore we can fix a small radius $\tau=\tau(p,q,\beta)$ so that 
$$
\osc_{B_\tau}\left(\nabla w\right)<\frac12.
$$
Since 
$$|\nabla w(0)|=\rho^{1-\beta}\underbrace{|\nabla v(x)|}_{\rho^{\beta-1}}=1,
$$
we must have $|\nabla w(y)|>1/2$ in $B_\tau$. Thus $w$ solves an equation which is uniformly elliptic with uniformly H\"older continuous coefficients in $B_\tau$. Recall also that $|w|\leq C$ in $B_1$ and hence also in $B_\tau$. Then, from Theorem 9.11 in \cite{GT01} and the Sobolev embedding, there are uniform $C^{1,\gamma}$-estimates available for $w$ with $\gamma =1-2/q$, so that
$$
\| w\|_{C^{1,\gamma}(B_{\tau/2 })}\leq B=B(p,q,\beta).
$$
In particular
$$
\sup_{y\in B_s}|w(y)-w(0)-(y-0)\cdot \nabla w(0)|\leq B|y-0|^{2-\frac2q}, 
$$
when $s<\tau/2$. In terms of $v$ this means
$$
\sup_{y\in B_s}\Big|\frac{v(x+\rho y)-v(x)}{\rho^\beta}-y\cdot \rho^{1-\beta}\nabla v(x)\Big|\leq B|y|^{2-\frac2q}.
$$
Write $z=x+\rho y$ and recall that $\beta\leq 2-\frac2q$. Then the above estimate reads
$$
\sup_{z\in B_{s\rho}}|v(z)-v(x)-(z-x)\cdot\nabla  v(x)|\leq B|y|^{2-\frac2q}\rho^\beta=B(s\rho)^\beta s^{2-\beta-\frac2q}\leq B(s\rho)^\beta,
$$
whenever 
$$
r=s\rho<\frac{\tau \rho}{2}=\frac{\tau}{2}|\nabla v(x)|^\frac{1}{\beta-1}.
$$
This is the same as saying that
$$
\sup_{z\in B_{r}}|v(z)-v(x)-(z-x)\cdot \nabla v(x)|\leq B|y|^{2-\frac2q}\rho^\beta=B(s\rho)^\beta s^{2-\beta-\frac2q}\leq Br^\beta,
$$
whenever 
$$
r<\frac{\tau \rho}{2}=\frac{\tau}{2}|\nabla v(x)|^\frac{1}{\beta-1}.
$$
It remains to verify estimate \eqref{eq:betaest} also when $r$ is in the interval $\tau\rho/2<r<\rho$. This is easy. Take such an $r$. Then estimate \eqref{eq:levelest} is available for the radius $\rho$ and we obtain
\begin{align*}
\sup_{z\in B_{r}}|v(z)-v(x)-(z-x)\cdot\nabla  v(x)|&\leq \sup_{z\in B_{\rho}}|v(z)-v(x)-(z-x)\cdot\nabla  v(x)|\\
&\leq (C+1)r^\beta\left(\frac{\rho}{r}\right)^\beta\leq \left(\frac{2}{\tau}\right)^\beta (C+1)r^\beta.
\end{align*}
Hence, we finally obtain the estimate \eqref{eq:betaest} for all $r<1/4$ with the constant
$$
L=\max \left(C+1,(C+1)\left(\frac{2}{\tau}\right)^\beta,B\right),
$$
which only depends on $p$, $q$ and $\beta$.
\end{proof}

We now conclude the proof of our main result.
\begin{proof}[~Proof of Theorem \ref{thm:main}] It is enough to prove the result for $\Omega = B_1$ and $K=B_{1/4}$. The normalized function
$$
\tilde v = \frac{v}{\max \left(\|v\|_{L^\infty(B_1)},\|h\|_{L^q(B_1)}^\frac{1}{p-1}\right)},
$$
satisfies the assumption of Theorem \ref{thm:mainprop}. Hence, the esimate \eqref{eq:betaest} holds true for $\tilde v$. Then the same estimate holds true for $v$ with $L$ replaced by 
$$
L\max \left(\|v\|_{L^\infty(B_1)},\|h\|_{L^q(B_1)}^\frac{1}{p-1}\right).
$$
This implies immediately
$$
|\nabla v(x)-\nabla v(y)|\leq 2 L\max \left(\|v\|_{L^\infty(B_1)},\|h\|_{L^q(B_1)}^\frac{1}{p-1}\right)|x-y|^{\beta -1}, 
$$
whenever $x,y\in B_{1/4}$. This ends the proof of the theorem.
\end{proof}
 \bibliographystyle{amsrefs}
\bibliography{ref}
\end{document}